\documentclass[a4paper,reqno,portrait,tbtags,11pt]{amsart}
\usepackage{graphicx}
\usepackage{amssymb}
\usepackage{amsmath}
\usepackage{amsfonts}

\newtheorem{theorem}{Theorem}[section]
\theoremstyle{plain}

\newtheorem{special case}{Special Case}

\newtheorem{definition}{Definition}[section]

\newtheorem{lemma}{Lemma}[section]

\numberwithin{equation}{section}
\input{tcilatex}

\begin{document}
\title[Brass-Stancu-Kantorovich operators on a Hypercube\ ]{%
Brass-Stancu-Kantorovich operators on a Hypercube$^{\ast }$\ }
\author{G\"{u}len Ba\c{s}canbaz-Tunca}
\address{Ankara University, Faculty of Science, Department of Mathematics,
Str. D\"{o}gol 06100, Be\c{s}evler, Ankara, Turkey}
\email{tunca@science.ankara.edu.tr}
\author{Heiner Gonska}
\address{University of Duisburg-Essen, Faculty of Mathematics, Forsthausweg
2, D-47057 Duisburg, Germany}
\email{heiner.gonska@uni-due.de and gonska.sibiu@gmail.com}
\keywords{Multivariate Kantorovich operator; Multivariate averaged modulus
of smoothness; Multivariate $K$-functional 
$^{\ast }$This paper is an extension of a talk given in ICATA 2022.}
\dedicatory{\textit{This study is dedicated to Professor Ioan Ra\c{s}a on
the occasion of his 70th birthday}}

\begin{abstract}
We deal with multivariate Brass-Stancu-Kantorovich operators depending on a
non-negative integer parameter and defined on the space of all Lebesgue
integrable functions on a unit hypercube. We prove $L^{p}$-approximation and
provide estimates for the $L^{p}$-norm of the error of approximation in
terms of a multivariate averaged modulus of continuity and of the
corresponding $L^{p}$-modulus.
\end{abstract}

\maketitle
\author{ }

\section{Introduction and Historical Notes}

The fundamental functions of the well-known Bernstein operators are defined
by 
\begin{equation}
p_{n,k}(x)=\left \{ 
\begin{array}{l}
{{\binom{n}{k}}x^{k}(1-x)^{n-k};\  \ 0\leq k\leq n} \\ 
\vspace{-0.3cm} \\ 
{0;\  \  \ k<0\  \text{or }k>n}%
\end{array}%
\right. ,\, \, \,x\in \lbrack 0,1].  \label{Fundamental Polynomials}
\end{equation}%
In \cite{Stancu-1981}, using a probabilistic method, Stancu generalized
Bernstein's fundamental functions as%
\begin{equation}
w_{n,k,r}(x):=\left \{ 
\begin{array}{ll}
\left( 1-x\right) p_{n-r,k}\left( x\right) ; & 0\leq k<r \\ 
\left( 1-x\right) p_{n-r,k}\left( x\right) +xp_{n-r,k-r}\left( x\right) ; & 
r\leq k\leq n-r \\ 
xp_{n-r,k-r}\left( x\right) ; & n-r<k\leq n%
\end{array}%
\right. ,\ x\in \lbrack 0,1],\text{ }  \label{Stancu fundamental func.}
\end{equation}%
where $r$\ is a non-negative integer parameter,\ $n$ is any natural number
such that $n>2r$,\ for which each $p_{n-r,k}\ $is given by (\ref{Fundamental
Polynomials}), and therefore, constructed and studied Bernstein-type
positive linear operators as 
\begin{equation}
L_{n,r}\left( f;x\right) :=\sum \limits_{k=0}^{n}w_{n,k,r}(x)f\left( \frac{k%
}{n}\right) ,\  \ x\in \lbrack 0,1],  \label{Stancu op. compact form}
\end{equation}%
for $f\in C[0,1]$.\ In doing so Stancu was guided by an article of Brass 
\cite{Brass-71}. This is further discussed by Gonska \cite{Gonska-2000}.
Among others, estimates in terms of the second order modulus of smoothness
are given there for continuous functions.

It is clear that for $x\in \lbrack 0,1]\ $Stancu's fundamental functions in (%
\ref{Stancu fundamental func.})\ satisfy\ 
\begin{equation*}
w_{n,k,r}(x)\geq 0\text{ and }\sum_{k=0}^{n}w_{n,k,r}(x)=1,
\end{equation*}%
hence the operators $L_{n,r}\ $can be expressed as

\begin{equation}
L_{n,r}\left( f;x\right) :=\sum \limits_{k=0}^{n-r}p_{n-r,k}\left( x\right) 
\left[ \left( 1-x\right) f\left( \frac{k}{n}\right) +xf\left( \frac{k+r}{n}%
\right) \right] ,  \label{Stancu op}
\end{equation}%
are defined for $n \ge r$ and satisfy the end point interpolation $%
L_{n,r}\left( f;0\right) =f\left( 0\right) ,\ L_{n,r}\left( f;1\right)
=f\left( 1\right) $. It thus seems to be justified to call the $L_{n,r}$
Brass-Stancu-Bernstein (BSB) operators.

In \cite{Stancu-1983} Stancu gave uniform convergence $\lim_{n\rightarrow
\infty }L_{n,r}\left( f\right) =$ $f\ $on $[0,1]\ $for $f\in C[0,1]$ and
presented an expression for the remainder $R_{n,r}(f;x)$ of the
approximation formula $f(x)=L_{n,r}(f;x)+R_{n,r}(f;x)\ $by means of second
order divided differences and also obtained an integral representation for
the remainder. Moreover, the author estimated the order of approximation by
the operators $L_{n,r}\left( f\right) $ via the classical modulus of
continuity. He also studied the spectral properties of $L_{n,r}$.

In the cases $r=0\ $and $r=1$, the operators $L_{n,r}\ $reduce to the
classical Bernstein operators\ $B_{n}$, i.e., 
\begin{equation*}
B_{n}\left( f;x\right) =\sum \limits_{k=0}^{n}p_{n,k}(x)f\left( \frac{k}{n}%
\right) .
\end{equation*}

What also has to be mentioned: Stancu himself in his 1983 paper observed
that "we can optimize the error bound of the approximation of the function $%
f $ by means of $L_{n,r}f$ if we take $r=0$ or $r=1$, when the operator $%
L_{n,r}$ reduces to Bernstein's." So there is a shortcoming.

Since Bernstein polynomials are not appropriate for approximation of
discontinuous functions (see \cite[Section 1.9]{LOR.}), by replacing the
point evaluations $f\left( \frac{k}{n}\right) \ $with the integral means
over small intervals around the knots $\frac{k}{n}$, Kantorovich \cite%
{Kantorovich} generalized the Bernstein operators as 
\begin{equation}
K_{n}\left( f;x\right) =\sum \limits_{k=0}^{n}p_{n,k}\left( x\right) \left(
n+1\right) \int \limits_{\frac{k}{n+1}}^{\frac{k+1}{n+1}}f\left( t\right)
dt,\  \  \ x\in \lbrack 0,1],\ n\in 
\mathbb{N}
,  \label{Berns-Kant.}
\end{equation}%
for Lebesgue integrable functions$\ f\ $on $[0,1]$.

On p. 239 of his mathematical memoirs \cite{Kant-1987} Kantorovich writes:
"While I was waiting for a student who was late, I was looking over vol.
XIII of Fundamenta Math. and saw in it a note from the Moscow Mathematician
Khlodovskii related to Bernstein polynomials. In it I first caught sight of
Bernstein polynomials, which he proposed in 1912 for an elementary proof of
the well known Weierstrass theorem ... I at once wondered if it is not
possible in these polynomials to change the values of the function at
certain points into the more stable average of the function in the
corresponding interval. It turned out that this was possible, and the
polynomials could be written in such a form not only for a continuous
function but also for any Lebesgue-summable function."

Lorentz \cite{LOR.} proved that $\lim \limits_{n\rightarrow \infty
}\left
\Vert K_{n}(f)-f\right \Vert _{p}=0, f\in L^{p}[0,1],\ 1\leq
p<\infty $.

There are a lot of articles dealing with classical Kantorovich operators,
and, in particular, their degree of approximation and the importance of
second order moduli of different types. See, e.g., the work of Berens and
DeVore \cite{BeDe1976}, \cite{BeDe1978}, Swetits and Wood \cite{SwWo1983}
and Gonska and Zhou \cite{GoZh1995}. It is beyond the scope of this note to
further discuss this matter. As further work on the classical case here we
only mention the 1976 work of M\"{u}ller \cite{Muller 1976}, Maier \cite%
{Maier}, and Altomare et al. \cite{Altomare et al. 2010}, see also the
references therein.

Similarly to Kantorovich operators Bodur et al. \cite{Bodur Bostanci B-T}
constructed a Kantorovich type modification of BSB operators as 
\begin{equation}
K_{n,r}\left( f;x\right) :=\sum_{k=0}^{n}w_{n,k,r}(x)\left( \left(
n+1\right) \int \limits_{\frac{k}{n+1}}^{\frac{k+1}{n+1}}f\left( t\right)
dt\right) ,\  \  \ x\in \lbrack 0,1],  \label{Knr*}
\end{equation}%
for $f\in L^{1}\left[ 0,1\right] $, where $r\ $is a non-negative integer
parameter,\ $n$ is a natural number such that $n>2r$ and $w_{n,k,r}(x)\ $are
given by (\ref{Stancu fundamental func.}). And, it was shown that \textit{If 
}$f\in L^{p}[0,1],\ 1\leq p<\infty $,\  \textit{then} $\lim
\limits_{n\rightarrow \infty }\left \Vert K_{n,r}(f)-f\right \Vert _{p}=0$.
In addition, it was obtained that each\textit{\ }$K_{n,r}\ $is variation
detracting as well \cite{Bodur Bostanci B-T}. Throughout the paper, we shall
call the operators $K_{n,r}\ $given by (\ref{Knr*})
"Brass-Stancu-Kantorovich", BSK operators.

Notice that from the definition of $w_{n,k,r}$, $K_{n,r}\left( f;x\right) \ $%
can be expressed as 
\begin{eqnarray}
&&K_{n,r}\left( f;x\right)  \label{Knr* explicit} \\
&=&\sum \limits_{k=0}^{n-r}p_{n-r,k}\left( x\right) \left( n+1\right) \left[
\left( 1-x\right) \int \limits_{\frac{k}{n+1}}^{\frac{k+1}{n+1}}f\left(
t\right) dt+x\int \limits_{\frac{k+r}{n+1}}^{\frac{k+r+1}{n+1}}f\left(
t\right) dt\right]  \notag
\end{eqnarray}%
and in the cases $r=0$ and $r=1\ $they reduce to the Kantorovich operators; $%
K_{n,0}=K_{n,1}=K_{n}\ $given by (\ref{Berns-Kant.}). Again they are defined
for all $n \ge r$.

MULTIVARIATE SITUATION

Some work has been done in the multivariate setting for BSB and BSK
operators. For the standard simplex this was done, e.g., by Yang, Xiong and
Cao \cite{Yang Xiong Cao} and Cao \cite{Cao}, For example, Cao proved that
multivariate Stancu operators preserve the properties of multivariate moduli
of continuity and obtained the rate of convergence with the help of
Ditzian-Totik's modulus of continuity.

In this work, motivated by the work Altomare et al. \cite{Altomare et al.
2017}, we deal with a multivariate extension of the BSK operators on a $d$%
-dimensional unit hypercube and we study $L^{p}$ -approximation by these
operators. For the rate of convergence we provide an estimate in terms of
the so called first order multivariate $\tau$-modulus, a quantity coming
from the Bulgarian school of Approximation Theory. Also, inspired by M\"{u}%
ller's approach in \cite{Muller}, we give estimates for differentiable
functions and such in terms of the $L^{p}$-modulus of smoothness, using
properties of the $\tau$-modulus. Here the work of Quak \cite{Qua1985}, \cite%
{Quak} was helpful.

\section{Preliminaries}

Consider the space $%
\mathbb{R}
^{d},\ d\in 
\mathbb{N}
$. Let $\left \Vert \mathbf{x}\right \Vert _{\infty }\ $denote the $\max $%
-norm\ of a point$\  \mathbf{x}=\left( x_{1},\ldots ,x_{d}\right) \in 
\mathbb{R}
^{d};$%
\begin{equation*}
\left \Vert \mathbf{x}\right \Vert _{\infty }:=\left \Vert \mathbf{x}\right
\Vert _{\max }=\max \limits_{i\in \left \{ 1,\ldots ,d\right \} }\left \vert
x_{i}\right \vert
\end{equation*}%
and let $\mathbf{1}$ denote the constant function $\mathbf{1}:%
\mathbb{R}
^{d}\rightarrow 
\mathbb{R}
\ $such that $\mathbf{1}\left( \mathbf{x}\right) =1\ $for $\mathbf{x}\in 
\mathbb{R}
^{d}$.$\ $And, for each $j=1,\ldots ,d$, let 
\begin{equation*}
pr_{j}:%
\mathbb{R}
^{d}\rightarrow 
\mathbb{R}%
\end{equation*}%
stand for the\ $j$th coordinate function defined for $\mathbf{x}\in 
\mathbb{R}
^{d}\ $by\ 
\begin{equation*}
pr_{j}\left( \mathbf{x}\right) =x_{j}.
\end{equation*}

\begin{definition}
A multi-index is a $d$-tuple $\mathbf{\alpha }=\left( \alpha _{1},\ldots
,\alpha _{d}\right) \ $of non-negative integers. Its norm (length) is the
quantity 
\begin{equation*}
\left \vert \mathbf{\alpha }\right \vert =\sum \limits_{i=1}^{d}\alpha _{i}.
\end{equation*}%
\ The differential operator $D^{\mathbf{\alpha }}\ $is defined by%
\begin{equation*}
D^{\mathbf{\alpha }}f=D_{1}^{\alpha _{1}}\cdots D_{d}^{\alpha _{d}}f,
\end{equation*}%
where $D_{i},\ i=1,\ldots ,d$,$\ $is the corresponding partial derivative
operator$\ $(see \cite[p. 335]{Bennett-Sharpley}).
\end{definition}

Throughout the paper $Q_{d}:=[0,1]^{d},\ d\in 
\mathbb{N}
$,\ will denote the $d$-dimensional unit hypercube and we consider the space%
\begin{equation*}
L^{p}\left( Q_{d}\right) =\left \{ f\ :Q_{d}\rightarrow 
\mathbb{R}
\mid f\text{ }p\text{-integrable on }Q_{d}\right \} ,\ 1\leq p<\infty \text{,%
}
\end{equation*}%
with the standard norm $\left \Vert .\right \Vert _{p}$. Recall the
following definition of the usual $L^{p}$-modulus of smoothness of first
order:

\begin{definition}
Let $f\in L^{p}\left( Q_{d}\right) ,\ 1\leq p<\infty $, $\mathbf{h}\in 
\mathbb{R}
^{d}\ $and $\delta >0$. The modulus of smoothness of the first order for the
function $f$\ and step $\delta \ $in $L^{p}$-norm is given by%
\begin{equation*}
\omega _{1}\left( f;\delta \right) _{p}=\sup_{0<\left \Vert \mathbf{h}\right
\Vert _{\infty }\leq \delta }\left( \int \limits_{Q_{d}}\left \vert f\left( 
\mathbf{x}+\mathbf{h}\right) -f\left( \mathbf{x}\right) \right \vert ^{p}d%
\mathbf{x}\right) ^{1/p}
\end{equation*}%
if$\  \mathbf{x},\mathbf{x}+\mathbf{h}\in Q_{d}$ \cite{Quak}.
\end{definition}

Let $M\left( Q_{d}\right) :=\left \{ f\mid f\text{ bounded and measurable on 
}Q_{d}\right \} $. Below, we present the concept of the first order averaged
modulus of smoothness\textit{.}

\begin{definition}
\textit{Let} $f\in M\left( Q_{d}\right) ,\  \mathbf{h}\in 
\mathbb{R}
^{d}\ $\textit{and} $\delta >0$. \textit{The multivariate averaged modulus
of smoothness, or} $\tau $-\textit{modulus, of the first order for function} 
$f$ \textit{and step} $\delta \ $\textit{in }$L^{p}$\textit{-norm is given by%
}%
\begin{equation*}
\tau _{1}\left( f,\delta \right) _{p}:=\left \Vert \omega _{1}\left(
f,.;\delta \right) \right \Vert _{p},\ 1\leq p<\infty ,
\end{equation*}%
\textit{where}%
\begin{equation*}
\begin{array}{l}
\omega _{1}\left( f,\mathbf{x};\delta \right) = \\ 
\sup \left \{ \left \vert f\left( \mathbf{t}+\mathbf{h}\right) -f\left( 
\mathbf{t}\right) \right \vert :\mathbf{t},\mathbf{t}+\mathbf{h}\in Q_{d},\
\left \Vert \mathbf{t}-\mathbf{x}\right \Vert _{\infty }\leq \frac{\delta }{2%
},\left \Vert \mathbf{t}+\mathbf{h}-\mathbf{x}\right \Vert _{\infty }\leq 
\frac{\delta }{2}\right \}%
\end{array}%
\end{equation*}%
\textit{is the multivariate local modulus of smoothness of first order for
the function }$f\ $\textit{at the point} $\mathbf{x}\in Q_{d}\ $\textit{and
for step} $\delta $. \cite{Quak}.
\end{definition}

For our future purposes, we need the following properties of first order
multivariate averaged modulus of smoothness:

For $f\in M\left( Q_{d}\right) ,\ 1\leq p<\infty \ $and $\delta ,\lambda
,\gamma \in 
\mathbb{R}
^{+}$, there hold

\begin{enumerate}
\item[$\protect \tau _{1})$] $\tau _{1}\left( f,\delta \right) _{p}\leq \tau
_{1}\left( f,\lambda \right) _{p}\ $for $0<\delta \leq \lambda ,$

\item[$\protect \tau _{2})$] $\tau _{1}\left( f,\lambda \delta \right)
_{p}\leq \left( 2\left \lfloor \lambda \right \rfloor +2\right) ^{d+1}\tau
_{1}\left( f,\delta \right) _{p}$, where $\left \lfloor \lambda
\right
\rfloor $ is the greatest integer that does not exceed $\lambda ,$

\item[$\protect \tau _{3})$] $\tau _{1}\left( f,\delta \right) _{p}\leq 2\sum
\limits_{\left \vert \mathbf{\alpha }\right \vert \geq 1}\delta
^{\left
\vert \mathbf{\alpha }\right \vert }\left \Vert D^{\mathbf{\alpha }%
}f\right
\Vert _{p},\  \alpha _{i}=0\ $or$\ 1$,\ if\ $D^{\mathbf{\alpha }%
}f\in L^{p}\left( Q_{d}\right) \ $for all multi-indices $\mathbf{\alpha }\ $%
with $\left \vert \mathbf{\alpha }\right \vert \geq 1$ and$\  \alpha _{i}=0\ $%
or$\ 1 $\ (see \cite{Popov-Khristov} or \cite{Quak}).
\end{enumerate}

For a detailed knowledge concerning averaged modulus of smoothness, we refer
to the book of Sendov and Popov \cite{Sendov-Popov}.

Now, consider the Sobolev space $W_{1}^{p}\left( Q_{d}\right) \ $of
functions $f\in L^{p}\left( Q_{d}\right) ,\ 1\leq p<\infty $,\ with
(distributional) derivatives$\ D^{\mathbf{\alpha }}f$ belong to $L^{p}\left(
Q_{d}\right) $,\ where $\left \vert \mathbf{\alpha }\right \vert \leq 1$,
with the seminorm\ 
\begin{equation*}
\left \vert f\right \vert _{W_{1}^{p}}=\sum \limits_{\left \vert \mathbf{%
\alpha }\right \vert =1}\left \Vert D^{\mathbf{\alpha }}f\right \Vert _{p}
\end{equation*}%
(see \cite[p. 336]{Bennett-Sharpley}). Recall that for all $f\in L^{p}\left(
Q_{d}\right) \ $the $K$-functional, in\ $L^{p}$-norm, is defined as 
\begin{equation}
K_{1,p}\left( f;t\right) :=\inf \left \{ \left \Vert f-g\right \Vert
_{p}+t\left \vert g\right \vert _{W_{1}^{p}}:g\in W_{1}^{p}\left(
Q_{d}\right) \right \} \  \  \  \left( t>0\right) .  \label{K-functional}
\end{equation}%
$K_{1,p}\left( f;t\right) \ $is equivalent with the usual first order
modulus of smoothness of $f$, $\omega _{1}\left( f;t\right) _{p}$;$\ $%
namely, there are positive constants $c_{1}$ and $c_{2}\ $such that 
\begin{equation}
c_{1}K_{1,p}\left( f;t\right) \leq \omega _{1}\left( f;t\right) _{p}\leq
c_{2}K_{1,p}\left( f;t\right) \  \  \  \left( t>0\right)  \label{equivalence}
\end{equation}%
holds for all $f\in L^{p}\left( Q_{d}\right) $ (see \cite[Formula 4.42 in p.
341]{Bennett-Sharpley}).

The following result due to Quak \cite{Quak} is an upper estimate for the $%
L^{p}$-norm of the approximation error by the multivariate positive linear
operators in terms of the first order averaged modulus of smoothness. Note
that this idea was used first by Popov for the univariate case in \cite%
{Popov}.

\begin{theorem}
\label{Quak Thm.}Let $L:M\left( Q_{d}\right) \rightarrow M\left(
Q_{d}\right) \ $be a positive linear operator that preserves the constants.
Then for every $f\in M\left( Q_{d}\right) \ $and $1\leq p<\infty $, the
following estimate holds:%
\begin{equation*}
\left \Vert L(f)-f\right \Vert _{p}\leq C\tau _{1}\left( f,\sqrt[2d]{A}%
\right) _{p},
\end{equation*}%
where $C$ is a positive constant and%
\begin{equation*}
A:=\sup \left \{ L\left( \left( pr_{i}\circ \psi _{\mathbf{x}}\right) ^{2};%
\mathbf{x}\right) :i=1,\ldots ,d,\  \mathbf{x}\in Q_{d}\right \} ,
\end{equation*}%
in which $\psi _{\mathbf{x}}\left( \mathbf{y}\right) :=\mathbf{y}-\mathbf{x}%
\ $for fixed $\mathbf{x}\in Q_{d}\ $and for every $\mathbf{y}\in Q_{d}\ $and 
$A\leq 1\ $\cite{Quak}.
\end{theorem}

\section{Multivariate BSK-Operators}

In this section, motivated by the works of Altomare et al. \cite{Altomare et
al. 2010} and Altomare et al. \cite{Altomare et al. 2017}, we consider the
multivariate extension of BSK-operators on $L^{p}\left( Q_{d}\right) $ and
study approximation properties of these operators in $L^{p}$-norm. We
investigate the rate of the convergence in terms of the first order $\tau $%
-modulus and the usual $L^{p}$-modulus of smoothness of the first order.

Let $r$ be a given non-negative integer. For any\ $n\in 
\mathbb{N}
\ $such that $n>2r,\  \mathbf{k}=\left( k_{1},\ldots ,k_{d}\right) \in
\left
\{ 0,\ldots ,n\right \} ^{d}\ $and $\mathbf{x}=\left( x_{1},\ldots
,x_{d}\right) \in Q_{d}$,\ we set 
\begin{equation}
w_{n,\mathbf{k},r}(\mathbf{x}):=\prod \limits_{i=1}^{d}w_{n,k_{i},r}(x_{i}),
\label{Compact Stancu Fund.}
\end{equation}%
where, $w_{n,k_{i},r}(x_{i})\ $is Stancu's fundamental function given by (%
\ref{Stancu fundamental func.}), written for each $i=1,\ldots ,d,$ $0\leq
k_{i}\leq n$\ and $x_{i}\in \lbrack 0,1]$.$\ $Thus, for $\mathbf{x}\in Q_{d}$%
,\ we have 
\begin{equation}
w_{n,\mathbf{k},r}(\mathbf{x})\geq 0\  \text{and }\sum \limits_{\mathbf{k}\in
\left \{ 0,\ldots ,n\right \} ^{d}}w_{n,\mathbf{k},r}(\mathbf{x})=1.
\label{Stancu dist.}
\end{equation}

For$\ f\in L^{1}\left( Q_{d}\right) \ $and $\mathbf{x}=\left( x_{1},\ldots
,x_{d}\right) \in Q_{d}\ $we consider the following multivariate extension
of the BSK-operators $K_{n,r}\ $given by (\ref{Knr*}):%
\begin{equation*}
K_{n,r}^{d}\left( f;\mathbf{x}\right) =\sum \limits_{k_{1},\ldots
,k_{d}=0}^{n}\prod \limits_{i=1}^{d}w_{n,k_{i},r}(x_{i})\int
\limits_{Q_{d}}f\left( \frac{k_{1}+u_{1}}{n+1},\ldots ,\frac{k_{d}+u_{d}}{n+1%
}\right) du_{1}\cdots du_{d}.
\end{equation*}%
Notice that from (\ref{Compact Stancu Fund.}), and denoting, as usual, any $%
f\in L^{1}\left( Q_{d}\right) \ $of $\mathbf{x}=\left( x_{1},\ldots
,x_{d}\right) \in Q_{d}\ $by $f\left( \mathbf{x}\right) =f\left(
x_{1},\ldots ,x_{d}\right) $,\ we can express these operators in compact
form as 
\begin{equation}
K_{n,r}^{d}\left( f;\mathbf{x}\right) =\sum \limits_{\mathbf{k}\in \left \{
0,\ldots ,n\right \} ^{d}}w_{n,\mathbf{k},r}(\mathbf{x})\int
\limits_{Q_{d}}f\left( \frac{\mathbf{k}+\mathbf{u}}{n+1}\right) d\mathbf{u}.
\label{Mvariate S-K}
\end{equation}%
It is clear that multivariate BSK-operators are positive and linear and the
cases $r=0$ and $1$ give the multivariate Kantorovich operators on the
hypercube $Q_{d}$, which can be captured from \cite{Altomare et al. 2010} as
a special case.

\begin{lemma}
\label{coordinate funcs.}For $\mathbf{x}\in Q_{d}$, we have%
\begin{eqnarray*}
K_{n,r}^{d}\left( \mathbf{1};\mathbf{x}\right) &=&1, \\
K_{n,r}^{d}\left( pr_{i};\mathbf{x}\right) &=&\frac{n}{n+1}x_{i}+\frac{1}{%
2\left( n+1\right) }, \\
K_{n,r}^{d}\left( pr_{i}^{2};\mathbf{x}\right) &=&\frac{n^{2}}{\left(
n+1\right) ^{2}}\left[ x_{i}^{2}+\left( 1+\frac{r\left( r-1\right) }{n}%
\right) \frac{x_{i}\left( 1-x_{i}\right) }{n}\right] \\
&&+\frac{3nx_{i}+1}{3\left( n+1\right) ^{2}},
\end{eqnarray*}%
for $i=1,\ldots ,d.$
\end{lemma}

Taking this lemma into consideration, by the well-known theorem of Volkov 
\cite{Volkov}, we immediately get that

\begin{theorem}
\label{Unif. Approx.}Let $r$ be a non-negative fixed integer and $f\in
C\left( Q_{d}\right) $. Then $\lim \limits_{n\rightarrow
\infty}K_{n,r}^{d}\left( f\right) =f\ $uniformly on $Q_{d}.$
\end{theorem}

Now, we need the following evaluations for the subsequent result: For $0\leq
x_{i}\leq 1,\ i=1,\ldots ,d$,\ we have%
\begin{eqnarray*}
\int \limits_{0}^{1}\left( 1-x_{i}\right) p_{n-r,k_{i}}\left( x_{i}\right)
dx_{i} &=&\binom{n-r}{k_{i}}\int \limits_{0}^{1}x_{i}^{k_{i}}\left(
1-x_{i}\right) ^{n-r-k_{i}+1}dx_{i} \\
&=&\frac{n-r-k_{i}+1}{\left( n-r+2\right) \left( n-r+1\right) }
\end{eqnarray*}%
when $0\leq k_{i}<r\ $and%
\begin{eqnarray*}
\int \limits_{0}^{1}x_{i}p_{n-r,k_{i}-r}\left( x_{i}\right) dx_{i} &=&\binom{%
n-r}{k_{i}-r}\int \limits_{0}^{1}x_{i}^{k_{i}-r+1}\left( 1-x_{i}\right)
^{n-k_{i}}dx_{i} \\
&=&\frac{k_{i}-r+1}{\left( n-r+2\right) \left( n-r+1\right) }
\end{eqnarray*}%
when $n-r<k_{i}\leq n$.\ Thus, from (\ref{Fundamental Polynomials}) and (\ref%
{Stancu fundamental func.}), it follows that%
\begin{equation}
\int \limits_{0}^{1}w_{n,k_{i},r}(x_{i})dx_{i}=\left \{ 
\begin{array}{ll}
\frac{n-r-k_{i}+1}{\left( n-r+2\right) \left( n-r+1\right) }; & 0\leq k_{i}<r
\\ 
\frac{n-2r+2}{\left( n-r+2\right) \left( n-r+1\right) }; & r\leq k_{i}\leq
n-r \\ 
\frac{k_{i}-r+1}{\left( n-r+2\right) \left( n-r+1\right) }; & n-r<k_{i}\leq n%
\end{array}%
\right. .  \label{integral of Stancu pol.}
\end{equation}%
Note that we can write the following estimates 
\begin{eqnarray}
n-r-k_{i}+1 &\leq &n-r+1\  \text{when }0\leq k_{i}<r,  \notag \\
n-2r+2 &\leq &n-r+1\  \text{when }r\leq k_{i}\leq n-r,  \notag \\
k_{i}-r+1 &\leq &n-r+1\  \text{when }n-r<k_{i}\leq n
\label{upper est. of the integrals}
\end{eqnarray}%
for each $i=1,\ldots ,d$,\ where in the middle term, we have used the
hypothesis $n>2r$. Making use of (\ref{upper est. of the integrals}), (\ref%
{integral of Stancu pol.}) and (\ref{Compact Stancu Fund.}), we obtain 
\begin{equation}
\int \limits_{Q_{d}}w_{n,\mathbf{k},r}(\mathbf{x})d\mathbf{x}=\prod
\limits_{i=1}^{d}\int \limits_{0}^{1}w_{n,k_{i},r}(x_{i})dx_{i}\leq \frac{1}{%
\left( n-r+2\right) ^{d}}.  \label{int.of Stancu pol. over hypercube}
\end{equation}

$L^{p}$-approximation by the sequence of the multivariate Stancu-Kantorovich
operators is presented in the following theorem.

\begin{theorem}
\label{Lp-approx.}Let $r$ be a non-negative fixed integer and $f\in
L^{p}\left( Q_{d}\right) ,\ 1\leq p<\infty $. Then\ $\lim
\limits_{n\rightarrow \infty }\left \Vert K_{n,r}^{d}(f)-f\right \Vert
_{p}=0 $.
\end{theorem}

\begin{proof}
Since the cases $r=0\ $and $1$ correspond to the multivariate Kantorovich
operators (see \cite{Altomare et al. 2010} or \cite{Altomare et al. 2017}),
we consider only the cases $r>1$, which is taken as fixed. From Theorem \ref%
{Unif. Approx.}, we obtain that $\lim \limits_{n\rightarrow \infty
}\left
\Vert K_{n,r}^{d}(f)-f\right \Vert _{p}=0\ $for any $f\in C\left(
Q_{d}\right) $.$\ $Since $C\left( Q_{d}\right) \ $is dense in $L^{p}\left(
Q_{d}\right) $, denoting the norm of the operator $K_{n,r}^{d}$ acting on $%
L^{p}\left( Q_{d}\right) \ $onto itself$\ $by $\left \Vert
K_{n,r}^{d}\right
\Vert $, it remains to show that there exists an $M_{r}$%
,\ where $M_{r}$ is a positive constant that maybe depends on $r$, such that 
$\left \Vert K_{n,r}^{d}\right \Vert \leq M_{r}\ $for all $n>2r$. Now, as in 
\cite[p.604]{Altomare et al. 2017}, we adopt the notation%
\begin{equation*}
Q_{n,\mathbf{k}}:=\prod \limits_{i=1}^{d}\left[ \frac{k_{i}}{n+1},\frac{%
k_{i}+1}{n+1}\right] \subset Q_{d};\  \bigcup \limits_{\mathbf{k}\in \left \{
0,\ldots ,n\right \} ^{d}}Q_{n,\mathbf{k}}=Q_{d}.
\end{equation*}%
Making use of the convexity of the function $\varphi \left( t\right)
:=\left
\vert t\right \vert ^{p},\ t\in 
\mathbb{R}
,\ 1\leq p<\infty \ $(see, e.g., \cite{Altomare}), and (\ref{Stancu dist.}),
for every $f\in L^{p}\left( Q_{d}\right) ,\ n>2r$, and $\mathbf{x}\in Q_{d}$%
, we obtain%
\begin{eqnarray*}
\left \vert K_{n,r}^{d}\left( f;\mathbf{x}\right) \right \vert ^{p} &\leq
&\sum \limits_{\mathbf{k}\in \left \{ 0,\ldots ,n\right \} ^{d}}w_{n,\mathbf{%
k},r}(\mathbf{x})\int \limits_{Q_{d}}\left \vert f\left( \frac{\mathbf{k}+%
\mathbf{u}}{n+1}\right) \right \vert ^{p}d\mathbf{u} \\
&=&\sum \limits_{\mathbf{k}\in \left \{ 0,\ldots ,n\right \} ^{d}}w_{n,%
\mathbf{k},r}(\mathbf{x})\left( n+1\right) ^{d}\int \limits_{Q_{n,\mathbf{k}%
}}\left \vert f\left( \mathbf{v}\right) \right \vert ^{p}d\mathbf{v}.
\end{eqnarray*}%
Taking (\ref{int.of Stancu pol. over hypercube}) into consideration, we
reach to%
\begin{equation*}
\int \limits_{Q_{d}}\left \vert K_{n,r}^{d}\left( f;\mathbf{x}\right) \right
\vert ^{p}d\mathbf{x}\leq \sum \limits_{\mathbf{k}\in \left \{ 0,\ldots
,n\right \} ^{d}}\left( \frac{n+1}{n-r+2}\right) ^{d}\int \limits_{Q_{n,%
\mathbf{k}}}\left \vert f\left( \mathbf{v}\right) \right \vert ^{p}d\mathbf{v%
}.
\end{equation*}%
Since $\sup \limits_{n>2r}\left( \frac{n+1}{n-r+2}\right) ^{d}=\left( \frac{%
2r+2}{r+3}\right) ^{d}:=M_{r}$\ for $r>1$, where $1<\frac{2r+2}{r+3}<2$, we
get 
\begin{equation*}
\int \limits_{Q_{d}}\left \vert K_{n,r}^{d}\left( f;\mathbf{x}\right) \right
\vert ^{p}d\mathbf{x}\leq M_{r}\int \limits_{Qd}\left \vert f\left( \mathbf{v%
}\right) \right \vert ^{p}d\mathbf{v},
\end{equation*}%
which implies that $\left \Vert K_{n,r}^{d}\left( f\right) \right \Vert
_{p}\leq M_{r}^{1/p}\left \Vert f\right \Vert _{p}$. Note that for the cases 
$r=0$ and $1$; we have $M_{r}=1\ $(see \cite{Altomare et al. 2017}).
Therefore, the proof is completed.
\end{proof}

\section{Estimates for the rate of convergence}

In \cite{Muller}, M\"{u}ller studied $L^{p}$-approximation by the sequence
of the Cheney-Sharma-Kantorovich operators (CSK). The author gave an
estimate for this approximation in terms of the univariate $\tau $-modulus
and moreover, using some properties of the $\tau $-modulus, he also obtained
upper estimates for the $L^{p}$-norm of the error of approximation for first
order differentiable functions as well as for continuous ones. In this part,
we show that similar estimates can also be obtained for $\left \Vert
K_{n,r}^{d}\left( f\right) -f\right \Vert _{p}\ $in the multivariate
setting. Our first result is an application of Quak's method in Theorem \ref%
{Quak Thm.}

\begin{theorem}
Let $r$ be a non-negative fixed integer, $f\in M\left( Q_{d}\right) \ $and $%
1\leq p<\infty $. Then 
\begin{equation}
\left \Vert K_{n,r}^{d}\left( f\right) -f\right \Vert _{p}\leq C\tau
_{1}\left( f,\sqrt[2d]{\frac{3n+1+3r\left( r-1\right) }{12\left( n+1\right)
^{2}}}\right) _{p}  \label{estimate via tau mod.}
\end{equation}%
for all $n\in 
\mathbb{N}
$ such that $n>2r$,$\ $where the positive constant $C$ does not depend on $f$%
.
\end{theorem}

\begin{proof}
According to Theorem \ref{Quak Thm.}; by taking $\psi _{\mathbf{x}}\left( 
\mathbf{y}\right) =\mathbf{y-x}\ $for fixed $\mathbf{x}\in Q_{d}\ $and for
every $\mathbf{y}\in Q_{d}$, and defining 
\begin{equation*}
A_{n,r}:=\sup \left \{ K_{n,r}^{d}\left( \left( pr_{i}\circ \psi _{\mathbf{x}%
}\right) ^{2};\mathbf{x}\right) :i=1,\ldots ,d,\  \mathbf{x}\in Q_{d}\right
\} ,
\end{equation*}%
where $\left( pr_{i}\circ \psi _{\mathbf{x}}\right)
^{2}=pr_{i}^{2}-2x_{i}pr_{i}+x_{i}^{2}\mathbf{1,}\ i=1,\ldots ,d$, we get
the following estimate%
\begin{equation*}
\left \Vert K_{n,r}^{d}\left( f\right) -f\right \Vert _{p}\leq C\tau
_{1}\left( f;\sqrt[2d]{A_{n,r}}\right)
\end{equation*}%
for any $f\in M\left( Q_{d}\right) $, under the condition that $A_{n,r}\leq
1 $. Now, applying the operators $K_{n,r}^{d}$ and making use of Lemma \ref%
{coordinate funcs.}, for every $i=1,\ldots ,d$ and $\mathbf{x}\in Q_{d}$, we
obtain 
\begin{eqnarray*}
K_{n,r}^{d}\left( \left( pr_{i}\circ \psi _{\mathbf{x}}\right) ^{2};\mathbf{x%
}\right) &=&\frac{n-1+r\left( r-1\right) }{\left( n+1\right) ^{2}}%
x_{i}\left( 1-x_{i}\right) +\frac{1}{3\left( n+1\right) ^{2}} \\
&\leq &\frac{n-1+r\left( r-1\right) }{4\left( n+1\right) ^{2}}+\frac{1}{%
3\left( n+1\right) ^{2}} \\
&=&\frac{3n+1+3r\left( r-1\right) }{12\left( n+1\right) ^{2}}
\end{eqnarray*}%
for all $n\in 
\mathbb{N}
$ such that$\ n>2r$, where $r\in 
\mathbb{N}
\cup \left \{ 0\right \} $. Therefore, since we have $n\geq 2r+1$, we take $%
r\leq \frac{n-1}{2}$ and obtain that $A_{n,r}\leq \frac{3n+1+3r\left(
r-1\right) }{12\left( n+1\right) ^{2}}\leq 1$ is satisfied, which completes
the proof.
\end{proof}

Now, making use of the properties $\tau 1)$-$\tau 3)\ $of the multivariate
first order $\tau $-modulus, we obtain

\begin{theorem}
\label{Estimate for smooth func.}Let $r$ be a non-negative fixed integer, $%
f\in L^{p}\left( Q_{d}\right) ,\ 1\leq p<\infty $,\ and $D^{\mathbf{\alpha }%
}f\in L^{p}\left( Q_{d}\right) \ $for all multi-indices $\mathbf{\alpha }\ $%
with $\left \vert \mathbf{\alpha }\right \vert \geq 1,\  \alpha _{i}=0$ or $1$%
. Then 
\begin{equation*}
\left \Vert K_{n,r}^{d}\left( f\right) -f\right \Vert _{p}\leq 2C_{r}\sum
\limits_{\left \vert \mathbf{\alpha }\right \vert \geq 1}\left( \frac{1}{%
\sqrt[2d]{n+1}}\right) ^{\left \vert \mathbf{\alpha }\right \vert }\left
\Vert D^{\mathbf{\alpha }}f\right \Vert _{p},
\end{equation*}%
for all $n\in 
\mathbb{N}
$ such that $n>2r$, where $C_{r}\ $is a positive constant depending on $r$.
\end{theorem}

\begin{proof}
Since $n>2r$, we immediately have $n+1\geq 2\left( r+1\right) $.\ Thus, the
term appearing inside the $2d$th root in the formula (\ref{estimate via tau
mod.}) can be estimated, respectively, for $r>1$, and $r=0,1$,$\ $as 
\begin{eqnarray*}
\frac{3n+1+3r\left( r-1\right) }{12\left( n+1\right) ^{2}} &=&\frac{%
3n+3+3r\left( r-1\right) -2}{12(n+1)^{2}} \\
&=&\frac{1}{n+1}\left[ \frac{1}{4}+\frac{3r\left( r-1\right) -2}{12(n+1)}%
\right] \\
&\leq &\frac{1}{n+1}\left[ \frac{1}{4}+\frac{3r\left( r-1\right) -2}{24(r+1)}%
\right] \\
&=&\frac{1}{n+1}\left[ \frac{3r^{2}+3r+4}{24(r+1)}\right]
\end{eqnarray*}%
and 
\begin{equation*}
\frac{3n+1}{12\left( n+1\right) ^{2}}=\frac{1}{n+1}\frac{3n+1}{4\left(
3n+3\right) }<\frac{1}{4\left( n+1\right) }.
\end{equation*}%
Now, defining 
\begin{equation*}
B_{r}:=\left \{ 
\begin{array}{cc}
\frac{3r^{2}+3r+4}{24(r+1)}; & r>1, \\ 
\frac{1}{4}; & r=0,1,%
\end{array}%
\right.
\end{equation*}%
and making use of the properties $\tau _{1})$-$\tau _{3})$ of $\tau $%
-modulus, from (\ref{estimate via tau mod.}), we arrive at 
\begin{eqnarray*}
\left \Vert K_{n,r}^{d}\left( f\right) -f\right \Vert _{p} &\leq &C\tau
_{1}\left( f,\sqrt[2d]{\frac{3n+1+3r\left( r-1\right) }{12\left( n+1\right)
^{2}}}\right) _{p} \\
&\leq &C\tau _{1}\left( f,\sqrt[2d]{B_{r}}\frac{1}{\sqrt[2d]{n+1}}\right)
_{p} \\
&\leq &C\left( 2\left \lfloor \sqrt[2d]{B_{r}}\right \rfloor +2\right)
^{d+1}\tau _{1}\left( f,\frac{1}{\sqrt[2d]{n+1}}\right) _{p} \\
&\leq &2C_{r}\sum \limits_{\left \vert \mathbf{\alpha }\right \vert \geq
1}\left( \frac{1}{\sqrt[2d]{n+1}}\right) ^{\left \vert \mathbf{\alpha }%
\right \vert }\left \Vert D^{\mathbf{\alpha }}f\right \Vert _{p},
\end{eqnarray*}%
where\ the positive constant $C_{r}\ $is defined as $C_{r}:=C\left(
2\left
\lfloor \sqrt[2d]{B_{r}}\right \rfloor +2\right) ^{d+1}.$
\end{proof}

For non-differentiable functions we have the following estimate in terms of
the first order modulus of smoothness, in $L^{p}$-norm.

\begin{theorem}
Let $r$ be a non-negative fixed integer\ and $f\in L^{p}\left( Q_{d}\right)
,\ 1\leq p<\infty $. \ Then%
\begin{equation*}
\left \Vert K_{n,r}^{d}\left( f\right) -f\right \Vert _{p}\leq
c_{2}C_{r,p}\omega _{1}\left( f;\frac{1}{\sqrt[2d]{n+1}}\right) _{p},
\end{equation*}%
where $\omega _{1}\ $is the first order multivariate modulus of smoothness
of $f$ and $C_{r,p}$ is a constant depending on $r\ $and $p$.
\end{theorem}

\begin{proof}
By Theorem \ref{Lp-approx.}, since $K_{n,r}^{d}\ $is bounded, with $%
\left
\Vert K_{n,r}^{d}\right \Vert _{p}\leq M_{r}^{1/p}$,\ for all $n\in 
\mathbb{N}
$ such that $n>2r$, we have $\left \Vert K_{n,r}^{d}\left( g\right)
-g\right
\Vert _{p}\leq \left( M_{r}^{1/p}+1\right) \left \Vert
g\right
\Vert _{p}\ $for $g\in L^{p}\left( Q_{d}\right) $. Moreover, from
Theorem \ref{Estimate for smooth func.}, we can write%
\begin{equation*}
\left \Vert K_{n,r}^{d}\left( g\right) -g\right \Vert _{p}\leq 2C_{r}\sum
\limits_{\left \vert \mathbf{\alpha }\right \vert \geq 1}\left( \frac{1}{%
\sqrt[2d]{n+1}}\right) ^{\left \vert \mathbf{\alpha }\right \vert }\left
\Vert D^{\mathbf{\alpha }}g\right \Vert _{p}
\end{equation*}%
for those $g\ $such that $D^{\mathbf{\alpha }}g\in L^{p}\left( Q_{d}\right) $%
, for all multi-indices $\mathbf{\alpha }\ $with $\left \vert \mathbf{\alpha 
}\right \vert \geq 1$ and$\  \alpha _{i}=0\ $or$\ 1$.\ Hence, for $f\in
L^{p}\left( Q_{d}\right) $, it readily follows that%
\begin{eqnarray*}
\left \Vert K_{n,r}^{d}\left( f\right) -f\right \Vert _{p} &\leq &\left
\Vert K_{n,r}^{d}\left( f-g\right) -\left( f-g\right) \right \Vert
_{p}+\left \Vert K_{n,r}^{d}\left( g\right) -g\right \Vert _{p} \\
&\leq &\left( M_{r}^{1/p}+1\right) \left \{ \left \Vert f-g\right \Vert
_{p}+2C_{r}\sum \limits_{\left \vert \mathbf{\alpha }\right \vert \geq
1}\left( \frac{1}{\sqrt[2d]{n+1}}\right) ^{\left \vert \mathbf{\alpha }%
\right \vert }\left \Vert D^{\mathbf{\alpha }}g\right \Vert _{p}\right \} .
\end{eqnarray*}%
Passing to the infimum for all $g\in W_{1}^{p}\left( Q_{d}\right) \ $in the
last formula, since the infimum of a superset does not exceed that of
subset, we obtain%
\begin{eqnarray}
&&\left \Vert K_{n,r}^{d}\left( f\right) -f\right \Vert _{p}  \notag \\
&\leq &\left( M_{r}^{1/p}+1\right) \inf \left \{ \left \Vert f-g\right \Vert
_{p}+\frac{2C_{r}}{\sqrt[2d]{n+1}}\sum \limits_{\left \vert \mathbf{\alpha }%
\right \vert =1}\left \Vert D^{\mathbf{\alpha }}g\right \Vert _{p}:g\in
W_{1}^{p}\left( Q_{d}\right) \right \}  \notag \\
&=&\left( M_{r}^{1/p}+1\right) \inf \left \{ \left \Vert f-g\right \Vert
_{p}+\frac{2C_{r}}{\sqrt[2d]{n+1}}\left \vert g\right \vert
_{W_{1}^{p}}:g\in W_{1}^{p}\left( Q_{d}\right) \right \}  \notag \\
&=&\left( M_{r}^{1/p}+1\right) K_{1,p}\left( f;\frac{2C_{r}}{\sqrt[2d]{n+1}}%
\right) ,  \label{Estimate via K-func}
\end{eqnarray}%
where $K_{1,p}\ $is the $K$-functional given by (\ref{K-functional}). The
proof follows from the equivalence (\ref{equivalence}) of the $K$-functional
and the first order modulus of smoothness in $L^{p}$-norm and the
non-decreasingness property of the modulus. Indeed, we get\ 
\begin{eqnarray}
K_{1,p}\left( f;\frac{2C_{r}}{\sqrt[2d]{n+1}}\right) &\leq &c_{2}\omega
_{1}\left( f;\frac{2C_{r}}{\sqrt[2d]{n+1}}\right) _{p}  \notag \\
&\leq &c_{2}\left( 2C_{r}+1\right) \omega _{1}\left( f;\frac{1}{\sqrt[2d]{n+1%
}}\right) _{p}.  \label{Upper estimate via mod of cont}
\end{eqnarray}%
Combining (\ref{Upper estimate via mod of cont}) with (\ref{Estimate via
K-func}) and defining $C_{r,p}:=\left( M_{r}^{1/p}+1\right) \left(
2C_{r}+1\right) $, where $M_{r}^{1/p}\ $and $C_{r}\ $are the same as in
Theorems \ref{Lp-approx.} and \ref{Estimate for smooth func.}, respectively,
we obtain the desired result.
\end{proof}

\end{document}